\documentclass[a4paper,11pt]{article}

\newcommand{\N}{\mathbb{N}}
\newcommand{\Z}{\mathbb{Z}}
\newcommand{\R}{\mathbb{R}}

\newcommand{\Q}{\mathbb{Q}}

\newcommand{\M}{\mathbb{M}}

\newcommand\norm[1]{\left\lVert#1\right\rVert}

\newcommand{\information}{{
  \bigskip
  \footnotesize
    
    \textbf{Jamerson Bezerra}:
    \textsc{FCUL-Faculdade de Ci\^encias da Universidade de Lisboa, Campo grande 1749-016, Lisboa} \par\nopagebreak
    \textit{E-mail:} \texttt{jdouglas@impa.br}
  
	\textbf{Carlos G. Moreira}:
    \textsc{IMPA-Instituto de Matem\'atica Pura e Aplicada, 22460-320, Rio de Janeiro } \par\nopagebreak
    \textit{E-mail:} \texttt{gugu@impa.br}
}}

\def\dim{\operatorname{dim}}

\def\diff{\operatorname{Diff}}
\def\per{\operatorname{Per}}
\def\diag{\operatorname{diag}}

\def\max{\operatorname{max}}

\def\per{\operatorname{Per}}

\def\interior{\operatorname{Int}}

\def\quand{\quad\text{and}\quad}

\def\GL{GL}

\usepackage[margin=1.4in]{geometry}

\usepackage{amsmath,amsthm,amssymb,amsfonts}
\usepackage{blindtext}
\usepackage{epsfig}
\usepackage{multicol}
\usepackage{xcolor}
\usepackage[hidelinks]{hyperref}
\usepackage{graphicx}
\usepackage{color}
\usepackage{psfrag}
\usepackage{multirow}
\usepackage{enumitem}
\usepackage{mathtools}
\usepackage{mathptmx}
\usepackage{newtxtext,newtxmath}
\usepackage{dirtytalk}
\usepackage[all]{xy}
\usepackage{xfrac}

\usepackage{caption}
\usepackage{subcaption}

\newtheorem{theorem}{Theorem}
\newtheorem{corollary}[theorem]{Corollary}
\newtheorem{proposition}[theorem]{Proposition}
\newtheorem{example}[theorem]{Example}

\newtheorem{lemma}[theorem]{Lemma}

\newtheorem{remark}[]{Remark}

\begin{document}

\title{Existence of periodic points with real and simple spectrum for diffeomorphisms in any dimension}
\author{Jamerson Bezerra and Carlos Gustavo Moreira}

\maketitle

\begin{abstract}
    We prove that for any $C^r$ diffeomorphism, $f$, of a compact manifold of dimension $d>2$, $1\leq r\leq \infty$, admitting a transverse homoclinic intersection, we can find a $C^1$-open neighborhood of $f$ containing a $C^1$-open and $C^r$-dense set of $C^r$ diffeomorphisms which have a periodic point with real and simple spectrum. We use this result to prove that $C^r$-generically among $C^r$ diffeomorphisms with horseshoes, we have density of periodic points with real and simple spectrum inside the horseshoe. As a corollary, we obtain that generically in the $C^1$-topology the unique obstruction to the existence of periodic points with real and simple spectrum are the Morse-Smale diffeomorphisms with all the periodic points admitting non-real eigenvalues.
\end{abstract}

\section{Introduction}

The periodic points are the ground basis to the study of general dynamical systems. The existence and classification of these special orbits led to many of the results in the field and are the basic structure used to have a global view of general dynamics.

For generic diffeomorphisms in compact manifolds the knowledge of the behaviour of orbits nearby periodic points conducts to comprehensive discoveries of robust phenomena which are exhaustively studied since the 80s. To mention some of these phenomena: robust homoclinic tangencies, existence of infinitely many sinks, blenders and others (see \cite{Ne1970}, \cite{PaVi1994}, \cite{BoDi1996}).

In \cite{PaVi1994}, Palis and Viana proved that for $C^2$-generic families of diffeomorphisms with a tangency associated with a sectionally dissipative hyperbolic periodic point we have periodic points with unique weakest expanding and contracting eigenvalues associated with a tangency. This is used to prove that $C^2$-near any diffeomorphism exhibiting a homoclinic tangency there exists a residual subset of an open set consisting of diffeomorphisms displaying infinitely many sinks. Romero, in \cite{Ro1995}, using the same structure of hyperbolic periodic orbits, extends the results in \cite{PaVi1994} to a context where only a weak version of sectional dissipativity is required.

In \cite{BoVi2004}, Bonatti and Viana proved a criterion for simplicity of Lyapunov spectrum for (real and complex) linear cocycles over hyperbolic shifts assuming a domination condition in the cocycles (fiber bunch condition) which implies the existence of linear holonomies. Using the holonomies they guarantee, after slight perturbation, the existence of periodic points such that the cocycles in these periodic points have real and simple spectrum. They claim that it is possible to use the same arguments for diffeomorphisms over hyperbolic basic sets in any regularity $0\leq r\leq \infty$. However, a careful inspection in their prove shows that it is necessary a good regularity of the holonomies which can be attained assuming domination conditions in the cocycles.

For locally constant cocycles taking values in the group of real symplectic matrices or taking values in the group of unitary matrices (real or complex), Matheus and Cambrainha in \cite{CaMa2016}, using algebraic methods, prove the existence of periodic points with real and simple spectrum for a full measure set of locally constant cocycles,.

Density of periodic points with real and simple spectra in non-trivial basic pieces for a $C^r$-residual set of diffeomorphisms in $3$-manifolds is obtained in \cite{BeRoVa2018}, by Bessa, Rocha and Varandas, using the ideas introduced in \cite{BoVi2004}. They apply this result to prove that in the $C^1$-topology either all periodic points of hyperbolic basic pieces for a diffeomorphism $f$ has simple Lyapunov spectrum $C^1$-robustly (in this case $f$ has a finest dominated splitting) or it can be $C^1$-approximated by an equidimensional cycle associated with periodic points with robust different signatures. 

In the $C^1$ topology, however, many perturbation techniques have been developed and a good description of the generic behaviour of the dynamics of diffeomorphisms has been attained in several works. One of these descriptions is given by Crovisier, in \cite{Cr2010}, where it is stated that any diffeomorphism in a compact manifold can be $C^1$-approximated by a Morse-Smale (finitely many periodic orbits with transverse intersection of their stable and unstable manifolds and zero topological entropy) or by one admitting a transverse homoclinic intersection (infinitely many periodic orbits and positive topological entropy)

The purpose of the present work is to analyse how this latter mechanism, namely, the existence of transverse homoclinic intersection, can led, for a $C^1$-open and $C^r$ dense set of diffeomorphisms, to plenty of periodic orbits with real and simple spectra in manifolds with any given dimension hence guaranteeing that $C^r$-generically among the diffeomorphisms with hyperbolic basic pieces we have density of the periodic points with real and simple spectra.

As a consequence, using the results described in \cite{Cr2010}, we prove that generically in the $C^1$-topology the unique obstruction to the existence of periodic orbits with real and simple spectra are the Morse-Smale diffeomorphisms with all the periodic points admitting non-real eigenvalues.

The work is divided as follows: in Section \ref{17621.1} we give the general definitions and precise statements of our results. In Section \ref{10621.1}, we introduce some notations of the linear algebra used in the upcoming sections. In Section \ref{28621.1}, we reduce the statement of the Theorem \ref{7621.1} to a linear algebra problem which is the most technical part of the work (see Proposition \ref{10621.2}). In Section \ref{28621.2}, we prove the main results using Proposition \ref{10621.2}. In the last section, \ref{17621.2}, we give the proof of Proposition \ref{10621.2}.

\section{Preliminary definitions and results}\label{17621.1}

Let $M$ be a smooth compact manifold of dimension $d > 2$. For any $r \in [1,\infty]$, we denote by $\diff^r(M)$ the set of $C^r$ diffeomorphisms of $M$ endowed with the natural $C^r$ topology.

Let $f:M\rightarrow M$ be a $C^r$ diffeomorphism of $M$. The \emph{spectrum} of a periodic point $p = f^l(p)$, denoted by $\sigma(df^l(p))$, is the set of eigenvalues of $df^l(p)$. We say that $p$ has \emph{real and simple spectrum} if all the eigenvalues of $df^l(p)$ are real and have different absolute values.

The diffeomorphism $f$ admits a \emph{transversal homoclinic intersection} if there exists a hyperbolic periodic point $p$ such that the stable manifold of $p$ and the unstable manifold of $p$ intersects each other transversely at some point $q\in M$.

\begin{theorem}\label{7621.1}
Fix $r\in [1,\infty]$. Let $f:M\rightarrow M$ be a $C^r$ diffeomorphism admitting a transversal homoclinic intersection associated with a hyperbolic periodic point $p\in M$. Giving an open neighborhood $U_p\subset M$ of $p$, there exists a $C^1$-open neighborhood of $f$, $\mathcal{U}(f)$, in $\diff^r(M)$ and a $C^1$-open and $C^r$-dense subset of $\mathcal{U}(f)$, $\mathcal{R}$, such that every $g\in \mathcal{R}$ admits a periodic point $p_g\in U_p$ with real and simple spectrum.
\end{theorem}


Let $f:M\rightarrow M$ be a $C^r$ diffeomorphism of $M$. A compact, $f$-invariant set $\Lambda\subset M$ is said to be \emph{hyperbolic} if there exist $df$-invariant continuous subbundles of $TM$, $E^s$ and $E^u$, and there exist constants, $C>0$ and $\lambda\in (0,1)$, such that $TM = E^s\oplus E^u$ and for every $x\in M$
\[
\norm{df^n(x)\cdot v}\leq C\lambda^n\norm{v}\quand \norm{df^{-n}\cdot u}\leq C\lambda^n\norm{u},
\]
for every $n\geq 0$, $v\in E^s(x)$ and $u\in E^u(x)$. The hyperbolic set $\Lambda$ is said to be a \emph{basic piece} if $f|_{\Lambda}$ is transitive and
\[
\overline{\per(f|_{\Lambda})} = \Lambda,
\]
where $\per(f|_{\Lambda})$ denote the set of periodic orbits of the diffeomorphism $f|_{\Lambda}$. Let $\Lambda$ be a basic piece for $f$ and $g:M\rightarrow M$ a diffeomorphism $C^1$ close to $f$, then there exists a basic piece for $g$, $\Lambda_g$, such that $f|_{\Lambda}$ is H\"older conjugated to $g|_{\Lambda_g}$. We call $\Lambda_g$ the \emph{hyperbolic continuation} of $\Lambda$ for $g$ (for more details see \cite{HaKa2002}).

\begin{corollary}\label{11621.2}
Fix $r\in [1,\infty]$. Let $\Lambda$ be a basic piece for $f\in \diff^r(M)$. Then, for every open set $U\subset M$, $U\cap \Lambda \neq \varnothing$, there exists a $C^1$-open neighborhood of $f$, $\mathcal{U}(f)$, in $\diff^r(M)$ and an $C^1$-open and $C^r$-dense subset of $\mathcal{U}(f)$, $\mathcal{R}$, such that every $g\in \mathcal{R}$, admits a periodic point $p_g\in \Lambda_g\cap U$ with real and simple spectrum. In particular, there exists a residual subset $\mathcal{G}$ of a $C^1$-open neighborhood $\mathcal{G}(f)\subset\diff^r(M)$ of $f$ such that if $g\in \mathcal{G}$ then the set of $g$-periodic points with real and simple spectrum is dense in $\Lambda_g$.
\end{corollary}

\begin{remark}
The arguments used in the prove of Theorem \ref{7621.1} and Corollary \ref{11621.2} can be easily adapted to guarantee that for generic H\"older linear cocycle over Markov shift we have a dense set of periodic points such that the cocycle has real and simple spectrum.
\end{remark}

The machinery used to perturb $C^1$ diffeomorphisms is well established in the literature. This allow us to improve the result in the $C^1$ topology to guarantee that the unique generic obstruction to the existence of the periodic points with real and simple spectrum are the Morse-Smale
\footnote{$f$ is \emph{Morse-Smale} if there are only finitely many periodic points all hyperbolic and such that for every pair of periodic points $p,q\in M$, all the intersections between $W^s(p)$ and $W^s(q)$ are transversal.}
diffeomorphisms with all the periodic points admitting a purely complex eigenvalue. 
\begin{corollary}\label{8621.1}
Any $C^1$ diffeomorphism can be $C^1$-approximated by diffeomorphism which is either Morse-Smale or has a transversal homoclinic orbit associated with a periodic point with real and simple spectrum.
\end{corollary}

\section{Notations and Matrix decomposition}\label{10621.1}

For $k,l \in \N$, we denote by $\M_{k\times l}(\R)$ the space of real matrices with $k$ rows, $l$ columns. When $k=l$ we denote $\M_{k\times k}(\R)$ simple by $\M_k(\R)$. $\GL_k(\R)$ stands for the space of invertible matrices in $\M_k(\R)$. We use the notation $A = \diag[A_1:\cdots:A_m]$ to describe a block-diagonal matrix, where $A_i \in \M_{k_i}(\R)$. For $A \in \M_k(\R)$, the set of eigenvalues of $A$, which it is called \emph{Spectrum} of $A$, is denoted by $\sigma(A)$. Define the inverse map $\iota:\GL_k(\R)\rightarrow \GL_k(\R)$ given by
\[
\iota(A) = A^{-1}.
\]

Let $i_1,\dots, i_m $ be positive integers with $i_1+\dots+i_m = d\geq 1$. Set, for each $j=1,\dots, m$, $\kappa_j = i_j+\dots+i_m$. We denote by $A_j: \M_{\kappa_j}(\R)\rightarrow \M_{i_j}(\R)$ and $D_j:\M_{\kappa_j}(\R)\rightarrow \M_{\kappa_{j+1}}(\R)$ the canonical projections associated with the decomposition $\R^d=\R^{i_1}\times\cdots\times\R^{i_m}$. Using this notation any $L\in \M_{\kappa_j}(\R)$ can be written in the block form as
\begin{align*}
    L=\left(
    \begin{array}{cc}
        A_j(L) & * \\
        * & D_j(L)
    \end{array}
    \right),
\end{align*}
for any $j=1,\dots,m$. Associated with the decomposition $\R^d=\R^{i_1}\times\cdots\times\R^{i_m}$ we define recursively the maps $D^{(j)}: \GL_d(\R)\rightarrow \M_{\kappa_{j+1}}(\R)$, $j=1,\dots, m-1$ by
\[
D^{(j)}(L) = D_j\circ D_{j-1}\circ\cdots\circ D_1(L)
\]
Observe that the above defined maps are analytic maps. The definition of such family of matrix valued maps will make sense in the proof of the Theorem \ref{7621.1} where a decomposition of the spectrum of periodic points will be necessary to obtain the desired result.
\begin{example}
As an example consider the three dimensional case with the decomposition $\R^3 = \R\times \R^2$. Then, we have the maps $A_1:\M_3(\R)\rightarrow \R$ and $D^{(1)} = D_1:\M_3(\R)\rightarrow \M_2(\R)$. So, if we take, for example, $L = \diag[\zeta_1:\zeta_2\cdot R_{\theta}]\in \GL_3(\R)$ we have $A_1(L) = \zeta_1$ and $D_1(L) = \zeta_2\cdot R_{\theta}$.
\end{example}
In the above example and throughout of this work we use the notation
\[
R_{\theta} = \left(
\begin{array}{cc}
    \cos2\pi\theta & -\sin2\pi\theta \\
    \sin2\pi\theta & \cos2\pi\theta  
\end{array}
\right)
\]
for the rigid rotation of angle $2\pi\theta$ with $\theta \in [0,1]$.

\section{Reduction by perturbations}\label{28621.1}
 
 In this section we show that by small $C^r$-perturbations we can assume that any diffeomorphism admitting a transversal homoclinic intersection satisfies a list of conditions which will be strongly used in Section \ref{28621.2} to proof Theorem \ref{7621.1}.
 
 Let $f:M\rightarrow M$ be a $C^r$ diffeomorphism of a compact manifold $M$ with $r\in [1,\infty]$ and $\dim(M)\geq 3$. Let $q \in M$ be a transversal homoclinic point associated with a hyperbolic periodic point $p\in M$ for $f$.
 
 We say that a periodic point $p$ of the diffeomorphism $f$, of period $\ell_0$, is \emph{non-resonant} if the eigenvalues of $df^{\ell_0}(p)$ do not satisfies any resonant condition, i.e., writing $\sigma(df^{\ell_0}(p)) = \{\lambda_1,\dots,\lambda_m\}$ we have that
 \begin{align}\label{8621.3}
 \lambda_1^{k_1}\cdots\lambda_m^{k_m} \neq 1,
 \end{align}
 for every $k = (k_1,\dots, k_m)\in \Z^m$.
 
\begin{enumerate}
    \item [1.] There are only countably many algebraic conditions such as those referred in \eqref{8621.3} to guarantee. For that reason we can always find a matrix which is close to $df^{\ell_0}(p)$ with non-resonant spectrum. So, making a $C^r$ perturbation, we assume that the hyperbolic periodic point $p$ of $f$ is non-resonant.

    \item [2.] Since $p$ is non-resonant, Using Sternberg's Linearization theorem (see \cite{St2010}), we can assume, by small perturbation of $f$ in the $C^r$-topology, that there exists linear $C^r$ coordinates in a neighborhood of $p$. In other words, there exists a small open neighborhood $U\subset M$ of $p$ and $C^r$ diffeomorphism $\varphi: U\rightarrow \R^n$ centered in $p$ such that
    \[
    d(\varphi\circ f^{\ell_0}\circ \varphi^{-1})(x) = d(\varphi\circ f^{\ell_0}\circ\varphi^{-1})(0),
    \]
    for every $x\in \varphi(U)$. We denote by $T$ the linear transformation $d(\varphi\circ f^{\ell_0} \circ\varphi^{-1}(0))$. We also abuse of the language and write $T = df^{\ell_0}(p)$.
    
    \item[3.] By small $C^r$-perturbation of $f$ we can assume that the linear transformation $T$ above defined has the following form:
    \begin{align*}
        T = \diag[T_1:\cdots: T_m],
    \end{align*}
    where for each $i \in \{1,\dots, m\}$ we have that
    \begin{itemize}
        \item $T_i \in \GL_1(\R) = \R^*$ or $T_i \in \GL_2(\R)$;
        \item In the case where $T_i\in \GL_2(\R)$ we have that $T_i = |\lambda_i|\cdot R_{\theta_i}$ for $\theta_i\in [0,1]\backslash \Q$ and $\lambda_i\in \sigma(T_i)$;
        \item If $j\in \{1,\dots,m\}$ with $i \geq j$, then for any $\lambda_i \in \sigma(T_i)$ and $\lambda_j\in \sigma(T_j)$ we have $|\lambda_i|>|\lambda_j|$.
    \end{itemize}
    This can be easily attained by slightly perturbation of the Jordan's canonical form of the matrix $T$.
\end{enumerate}

Consider the points $q_1, q_2\in M$ in the orbit of the transverse homoclinic point $q$, $f^{\ell_0 N_0}(q_2) = q_1$, such that $f^{\ell_0 n}(q_1)$ and $f^{-\ell_0 n}(q_2)$ belong to the open neighborhood with linear coordinates $U$, for every non negative integer $n$. In other words, $f^{\ell_0 N_0}(q_2) = q_1$ and
\[
f^{\ell_0 n}(q_1)\in W^s_{loc}(p)\cap U \quand f^{-\ell_0 n}(q_2)\in W^u_{loc}(p)\cap U,
\]
for every $n\geq 0$.

For $n$ large enough we can take a neighborhood $V(n)\subset U$ of $p$ and $q_1$ such that $q_2\in \interior(f^{\ell_0n}(V(n)))$. So, the $f^{\ell_0(n + N_0)}$-invariant set
\[
\Lambda(n) = \bigcap_{k\in \Z}f^{k\ell_0(n + N_0)}(V(n)),
\]
is a basic piece for $f^{\ell_0(n+N_0)}$ such that $f^{\ell_0(n+N_0)}|_{\Lambda(n)}$ is conjugated to a two-sided shift in the space $\{0,1\}^{\Z}$. Set $p_n$ as the unique $f^{\ell_0(n+N_0)}$-fixed point in $\Lambda(n)\backslash\{p\}$. Hence, for every $n$ large enough, we have that
\[
f^{\ell_0 j}(p_n)\in U
\]
for every $j = 0,\dots, n$. Note also that $V(n)$ can be build in a fashion to guarantee that $p_n$ converges to $q_1$ and so $f^{\ell_0 n}(p_n) = f^{-\ell_0N_0}(p_n)$ converges to $f^{-\ell N_0}(q_1) = q_2$.

We fix the decomposition $\R^d = \R^{i_1}\times\cdots\times\R^{i_m}$ where each $i_k$ is chosen such that $T_k \in \GL_{i_k}(\R)$. So, $i_k\in \{1,2\}$ for every $k=1,\dots,m$ and $i_1+\cdots+i_m = d$. We also fix the notation, for each $k = 1,\cdots, m$, $\kappa_k = i_k+\cdots+i_m$.
\begin{enumerate}
    \item[4.] We can perform an arbitrarily small $C^r$ perturbation of $f$, (modifying $f$ only in a neighborhood of $q_2$), such that $L := df^{\ell_0 N_0}(q_2)$ (in coordinates) satisfies, recursively, the following conditions (look at the notations introduced in the Section \ref{10621.1}):
    \begin{itemize}
        \item $L = D^{(0)}(L) \in \GL_d(\R) = \GL_{\kappa_1}(\R)$ and $A_1\circ D^{(0)}(L) = A_1(L) \in \GL_{i_1}(\R)$ has different singular values;
        \item For every $1\leq j \leq m-1$, $A_{j+1}\circ\iota\circ D^{(j)}\circ \iota(L) \in \GL_{i_{j+1}}(\R)$ has different singular values and $D^{(j)}\circ \iota(L) = D_j\circ D_{j-1}\circ\cdots\circ D_1\circ\iota(L) \in \GL_{\kappa_{j+1}}(\R)$.
    \end{itemize}
\end{enumerate}

\begin{remark}
In the case of a matrix in $\GL_1(\R) = \R^*$ does not make sense to ask different singular values. So, it is intrinsically assumed that we are asking different singular values only for those indices $1\leq k\leq m$ such that $i_k = 2$.
\end{remark}

To simplify notation, set $F = f^{\ell_0}$. Using the above constructed local coordinates around $p$ and taking $n$ large enough we have that (remember that $F$-period of $p_n$ is $n + N_0$ and that $dF(0) = T$)
\begin{align*}
    dF^{n + N_0}(p_n) &= dF^{N_0}(F^{n}(p_n))\cdot dF^{n}(p_n)\\
    &= L_n\cdot T^{n},
\end{align*}
 where $L_n = dF^{N_0}(F^{n}(p_n))$ converges to $L = dF^{N_0}(q_2)$, since $F^{n}(p_n) = F^{-N_0}(p_n)$ converges to $F^{-N_0}(q_1) = q_2$.

We summarize the above discussion in the next proposition.
\begin{proposition}\label{10621.2}
 Fix $d\geq 3$. Consider $L\in \GL_d(\R)$ and $T = \diag[T_1:\cdots:T_m]\in \GL_d(\R)$ satisfying the following conditions:
 \subparagraph{Conditions on T:}
 \begin{itemize}
    \item $T_i \in \GL_1(\R) = \R^*$ or $T_i \in \GL_2(\R)$, for every $1\leq i\leq m$;
    \item If $I = \{j\in \{1,\dots,\kappa\};\ i_j = 2\}$, then, for each $i\in I$, $T_i = |\lambda_i|\cdot R_{\theta_i}$ with $\theta_i\in [0,1]$ and $\theta = (\theta_i)_{i\in I}\in \R^{\#I}$ is rationally independent.
    \item If $i,j\in \{1,\dots,m\}$ with $i \geq j$, then for any $\lambda_i \in \sigma(T_i)$ and $\lambda_j\in \sigma(T_j)$ we have $|\lambda_i|>|\lambda_j|$;
    \end{itemize}
    \subparagraph{Conditions on L:}
    \begin{itemize}
    \item $L = D^{(0)}(L) \in \GL_d(\R) = \GL_{\kappa_1}(\R)$ and $A_1\circ D^{(0)}(L) = A_1(L) \in \GL_{i_1}(\R)$ has different singular values;
    \item For every $1\leq j \leq m-1$,$A_{j+1}\circ\iota\circ D^{(j)}\circ \iota(L) \in \GL_{i_{j+1}}(\R)$ has different singular values and $D^{(j)}\circ \iota(L) \in \GL_{\kappa_{j+1}}(\R)$.
 \end{itemize}
 Then, for every sequence $(L_n)_n\subset \GL_d(\R)$ converging to $L$ and for every arithmetic projection $(an+b)_{n\in \N}$ there exists a subsequence $(n_k)_k$ such that $L_{n_k}\cdot T^{an_k+b}$ has real and simple spectrum for every $k\in \N$.
\end{proposition}
Taking
\[
T = df^{\ell_0}(p), L = dF^{N_0}(q_2) \quand L_n = dF^{N_0}(F^{n}(p_n)),
\]
we can use the reductions described in this section and the Proposition \ref{10621.2} to prove the density part of the Theorem \ref{7621.1}. The above Proposition \ref{10621.2} is the main technical result of this work and its proof will be described in the Section \ref{17621.2}.

\section{Proof the results}\label{28621.2}

In this section we give the proof of the results stated in the Section \ref{17621.1}.

\paragraph{Proof of Theorem \ref{7621.1}:} Fix $r\in[1,\infty]$. Let $f_0:M\rightarrow M$ be a $C^r$ diffeomorphism admitting a transverse homoclinic intersection associated with a hyperbolic periodic point $p_0$ and let $U_{p_0}\subset M$ be an open neighborhood of $p_0$ as in the statement of the Theorem \ref{7621.1}. Consider the $C^1$-open neighborhood $\mathcal{U}(f_0)\subset \diff^r(M)$ of $f_0$ such that every diffeomorphism in $\mathcal{U}(f_0)$ admits a transverse homoclinic intersection associated with a periodic point in $U_{p_0}$. The existence of such $\mathcal{U}(f_0)$ can be justified using transversality.

Our goal is to guarantee that the set, $\mathcal{R}$, of diffeomorphisms $f\in \mathcal{U}(f_0)\subset \diff^r(M)$ with a periodic point $p\in U_{p_0}$ with simple spectrum is $C^1$-open and $C^r$-dense. Due to the fact that the set of matrices with real and simple spectrum consists of an open set in the set of matrices we can easily see that $\mathcal{R}$ is $C^1$-open.

To prove density, consider any $f\in \mathcal{U}(f_0)$ and let $p\in U_{p_0}$ be a hyperbolic periodic point for $f$, of period $\ell_0$, such that the stable manifold of $p$ and the unstable manifold of $p$ intersects transversely at some point $q\in M$.

Following the reductions described in the Section \ref{10621.1} we are able to assume, after arbitrarily small $C^r$ perturbations, that there exists $U$ linearising coordinates in an open neighborhood $U\subset M$ around $p$ (which can be assume to be a subset of $U_{p_0}$) such that $T := df^{\ell_0}(p) = \diag[T_1:\cdots:T_m]$ satisfying the conditions on Proposition \ref{10621.2}. Taking points $q_1, q_2 \in U$ in the orbit of the transverse homoclinic point $q$, with $f^{\ell_0 N_0}(q_2) = q_1$ we can find a sequence of periodic points $(p_n)_n\subset U$ of period $\ell_0(n+N_0)$ which converges to $q_1$. Taking $L = df^{\ell_0 N_0}(q_2)$ (in $U$ coordinates), which satisfies the conditions of Proposition \ref{10621.2}, and considering $n$ sufficient large we can write (in coordinates)
\[
df^{\ell_0(n + N_0)}(p_n) = df^{\ell_0 N_0}(f^{\ell_0 n}(p_n))\cdot df^{\ell_0 n}(p_n) = L_n\cdot T^{n},
\]
with $L_n$ converging to $L$. Therefore, we are in the conditions to apply Proposition \ref{10621.2} to guarantee that there exists a subsequence $(n_k)_k\subset \N$ such that
\[
df^{\ell_0 (n_k+N_0)}(p_{n_k}) = L_{n_k}T^{n_k}
\]
has real and simple spectrum. In other words, arbitrarily $C^r$-close to $f$ we can find a diffeomorphism which has infinitely many periodic point in $U_{p_0}$ with real and simple spectrum. Hence, $\mathcal{R}\subset \mathcal{U}(f_0)$ is dense in $\mathcal{U}(f_0)$.\qed

\paragraph{Proof of Corollary \ref{11621.2}:} Let $f_0\in \diff^r(M)$ be a diffeomorphism with a basic set $\Lambda_0$ and consider an open subset $U\subset M$ such that $U\cap \Lambda_0 \neq \varnothing$. It is well known that the set of periodic point in $\Lambda_0$ with transverse homoclinic intersection is dense in $\Lambda_0$. Hence, there exists a periodic point $p\in U\cap \Lambda_0$ which has a transverse homoclinic intersection. Using Theorem \ref{7621.1} we can find a neighborhood $\mathcal{U}(f_0)$, of $f_0$, and a $C^1$-open $C^r$-dense set $\mathcal{R}(f_0)\subset \mathcal{U}(f_0)$ such that any diffeomorphism in $\mathcal{R}(f_0)$ has a periodic point inside $U$ with real and simple spectrum.

Now we prove that generically in some $C^1$-open neighborhood of $f_0$ in $\diff^r(M)$ we have that the hyperbolic continuation of $\Lambda$ has a dense subset of periodic orbits with real and simple spectrum. Let $\mathcal{G}(f_0)\subset \diff^r(M)$ be a $C^1$-open neighborhood of $f_0$ such that any diffeomorphism $f\in \mathcal{G}(f_0)$ admits a hyperbolic continuation $\Lambda_f$ of $\Lambda_0$. For each any $n\geq 1$ consider the finite cover, $\mathcal{B}_n$, of $\Lambda_0$ by balls with radius $1/n$. Shrinking $\mathcal{G}(f_0)$ if necessary, we can assume that $\mathcal{B}_n$ is a cover of $\Lambda_f$ for every $f\in \mathcal{U}(f_0)$ and every $n\geq 1$.

By the first part of the Corollary \ref{11621.2} we can find $\mathcal{G}_1\subset \mathcal{G}(f_0)$, $C^1$-open and $C^r$-dense in $\mathcal{G}(f_0)$, such that any $g\in \mathcal{G}_1$ has periodic points with real and simple spectrum in $B\cap\Lambda_g$ for any element $B\in \mathcal{B}_1$ with $B\cap \Lambda_g\neq \varnothing$.

Now, assume that we are able to find $\mathcal{G}_n\subset \mathcal{G}(f_0)$, $C^1$-open and $C^r$-dense in $\mathcal{G}(f_0)$, such that any $g\in \mathcal{G}_n$ has periodic points with simple and real spectrum in $B\cap\Lambda_g$ for any element $B\in \mathcal{B}_n$ with $B\cap\Lambda_g\neq\varnothing$. Then, for each $g\in \mathcal{G}_n$ we can apply the first part of the Corollary \ref{11621.2} to find a $C^1$-open neighborhood $\mathcal{R}_{n+1}(g)\subset \mathcal{G}_n$ of $g$ in $\diff^r(M)$ and a subset $\mathcal{G}_{n+1}(g)\subset \mathcal{R}_{n+1}(g)$ such that any $g'\in \mathcal{G}_{n+1}(g)$ has periodic points with real and simple spectrum in $B\cap\Lambda_g$ for any element $B\in \mathcal{B}_{n+1}$ with $B\cap \Lambda_{g'}\neq\varnothing$. Define
\begin{align*}
    \mathcal{G}_{n+1} := \bigcup_{g\in \mathcal{G}_n}\mathcal{G}_{n+1}(g)\subset \bigcup_{g\in\mathcal{G}_n}\mathcal{R}_{n+1}(g)\subset \mathcal{G}_n.
\end{align*}
Note that $\mathcal{G}_{n+1}$ is $C^1$-open (union of open sets) and $C^r$-dense in $\bigcup_{g\in \mathcal{G}_n}\mathcal{R}_{n+1}(g)$. Since,
\begin{align*}
    \overline{\bigcup_{g\in \mathcal{G}_n}\mathcal{R}_{n+1}(g)} = \overline{\mathcal{G}_n},
\end{align*}
we have that $\mathcal{G}_{n+1}$ is $C^1$-open and $C^r$-dense in $\mathcal{G}(f_0)$.

We build then a decrease sequence of $C^1$-open and $C^r$-dense subsets of $\mathcal{G}(f_0)$. Consider the residual set
\begin{align*}
    \mathcal{G} = \bigcap_{n\geq 1}\mathcal{G}_n,
\end{align*}
and observe that if $g\in \mathcal{G}$, then $\Lambda_g$ has a dense set of periodic points with real and simple spectrum.\qed

\paragraph{Proof of Corollary \ref{8621.1}:} Let $f_0\in \diff^1(M)$ which cannot be approximated by Morse-Smale diffeomorphisms. Then, we can find $f\in \diff^1(M)$ close to $f_0$ such that $f$ admits a periodic point with a transverse homoclinic intersection (see \cite{Cr2010}). Applying Theorem \ref{7621.1} with $r = 1$ we can find $\tilde{f}\diff^1(M)$, $C^1$-close to $f$, such that $\tilde{f}$ has a periodic point with real and simple spectrum.\qed

\section{Proof of the Propositions \ref{10621.2}:}\label{17621.2}

The following lemma plays a major rule in the proof of the Proposition \ref{10621.2} allowing us to decompose the spectrum of $LT^n$, for $n$ sufficient large, in a disjoint union sets which are spectrum of $i$-dimensional matrices with $i\in \{1,2\}$. This reduction led us to analyse the problem only in small dimensional blocks.

In what follows for any function $\psi: \Gamma_1\rightarrow \Gamma_2$, $G(\psi)\subseteq \Gamma_1\times\Gamma_2$, denotes the graph of the map $\psi$. 
\begin{lemma}\label{11621.3}
Fix positive integers $\kappa_1$ and $\kappa_2$, $\kappa_1+\kappa_2 = d$. Consider the linear projections $A,B,C$ and $D$ defined, for every $J\in \M_{\kappa_1\times\kappa_2}(\R)$, by the following expression
\begin{align*}
    J = \left(
    \begin{array}{cc}
        A(J) & B(J) \\
        C(J) & D(J)
    \end{array}
    \right).
\end{align*}
where the blocks respects the fixed decomposition $\R^d = \R^{\kappa_1}\times\R^{\kappa_2}$. Let $V\in \GL_d(\R)$ and $J_0\in \GL_d(\R)$ satisfying the following conditions:
\begin{itemize}
    \item $A(V), A(J_0)\in \GL_{\kappa_1}(\R)$ and $D(V), D(J_0^{-1}) \in \GL_{\kappa_2}(\R)$;
    \item $\norm{D(V)}< \norm{A(V)^{-1}}^{-1}$.
\end{itemize}
Then, for every given $\delta>0$, there exist $\beta>0$, $\gamma>0$ and $n_0\in \N$ such that for every $J\in\GL_d(\R)$, with $\norm{J- J_0}<\beta$, and for every $n\geq n_0$, there exist linear $JV^n$-invariant functions $\xi_{n,J}:\R^{\kappa_1}\rightarrow \R^{\kappa_2}$ and $\eta_{n,J}:\R^{\kappa_2}\rightarrow \R^{\kappa_1}$ such that
\begin{enumerate}
    \item $\norm{\xi_{n,J}}\leq \gamma$ and $\norm{\eta_{n,J}}\leq \gamma\left(\lVert{P^{-1}}\rVert\lvert{Q}\rVert\right)^n$
    \item $G(\xi_{n,J}) \oplus G(\eta_{n,J}) = \R^d$;
    \item $\norm{\left(\pi_{\xi_{n,J}}\cdot JV^n|_{G(\xi_{n,J})}\cdot\pi^{-1}_{\xi_{n,J}}\right)\cdot A(V)^{-n} - A(J_0)}<\delta$;
    \item $\norm{D(V)^n\cdot \left(\pi_{\eta_{n,J}}\cdot (JV^n)^{-1}|_{G(\eta_{n,J})}\cdot \pi^{-1}_{\eta_{n,j}}\right) - D(J_0^{-1})} < \delta$,
\end{enumerate}
where $\pi_{\xi_{n,J}}: G(\xi_{n,J})\rightarrow \R^{\kappa_1}$ and $\pi_{\eta_{n,J}}:G(\eta_{n,J})\rightarrow \R^{\kappa_2}$ are the canonical projections.
\end{lemma}
\begin{remark}
In a restrict form and without proof, the Lemma \ref{11621.3} appeared in \cite{PaVi1994} where the uniformity of $n_0$ in a neighborhood $J_0\in \GL_d(\R)$ is not required.
\end{remark}
\begin{proof}
Fix $\delta >0$. By continuity of the maps $A$ and $D\circ \iota$, there exists $\beta = \beta(\delta, J_0)>0$ such that
\begin{align}\label{14621.1}
    \norm{A(J) - A(J_0)} < \delta/2 \quand \norm{D(J^{-1}) - D(J_0^{-1})} < \delta/2,
\end{align}
for every $J\in \GL_d(\R)$ with $\norm{J-J_0}< \beta$. Using also the continuity of the projections $B, C$ and $D$ there exists $\alpha = \alpha (\beta)>0$ such that
\begin{align}
    \max_{\norm{J-J_0}\leq \beta}\left\{
    \norm{D(J)}, \norm{A(J)^{-1}}, \norm{B(J)},\norm{C(J) A(J)^{-1}}
    \right\}
    \leq
    \alpha,
\end{align}
and
\begin{align}\label{14621.2}
    \max_{\norm{J-J_0}\leq \beta}\left\{
    \norm{A(J^{-1})}, \norm{C(J^{-1})}, \norm{D(J^{-1})^{-1}}, \norm{B(J^{-1})D(J^{-1})^{-1}}
    \right\}
    \leq
    \alpha
\end{align}
For each $n\in \N$ and $J\in \GL_d(\R)$ with $\norm{J-J_0}< \beta$, consider the (non-linear) operator $\phi_{n,J}:\M_{\kappa_2\times\kappa_1}(\R)\rightarrow \M_{\kappa_2\times\kappa_1}(\R)$ given by 
\begin{align*}
    \phi_{n,J}(u) = C(J)A(J)^{-1} + \left(
    D(J) - uB(J)
    \right)\cdot D(V)^nuA(V)^{-n}A(J)^{-1}
\end{align*}
Set $\mathcal{B}_{\gamma}$ as the set of maps $u\in \M_{\kappa_2\times\kappa_1}(\R)$ such that $\norm{u} \leq \gamma$, where $\gamma$ is choose such that $\gamma > \alpha$.

\paragraph{Claim 1:} We claim that $\phi_{n,J}(\mathcal{B}_{\gamma})\subset \mathcal{B}_{\gamma}$ for any $n$ large enough. Indeed, consider $n_1 = n_1(||D(V)||\cdot||A(V)^{-1}||, \gamma, \alpha)  \in\N$ given by
\[
n_1 = \left \lceil{
\frac{1}{\log\left(||D(V)||\cdot||A(V)^{-1}||\right)}\log\left(\frac{\gamma - \alpha}{\alpha^2\gamma(1+\gamma)}\right)
}\right \rceil,
\]
which is well defined due to the fact that $||D(V)||\cdot ||A(V)^{-1}|| < 1 $. Then, for $n\geq n_1$ and $u\in \mathcal{B}_{\alpha}$ we have
\begin{align*}
    \norm{\phi_{n,J}(u)}
    &\leq
    \norm{C(J)A(J)^{-1}} + \norm{D(J) - uB(J)}\norm{u}\norm{A(J)^{-1}}\left(||D(V)||\cdot||A(V)^{-1}||\right)^n\\
    &\leq
    \alpha + \alpha^2\gamma(1 + \gamma)\left(||D(V)||\cdot ||A(V)^{-1}||\right)^n\\
    &\leq
    \gamma.
\end{align*}
Hence, the map $\phi_{n,J}$ preserves $\mathcal{B}_{\gamma}$ for every $J\in\GL_d(R)$ with $\norm{J-J_0}< \beta$ and for every $n\geq n_1$.

\paragraph{Claim 2:} The map $\phi_{n,J}|_{\mathcal{B}_{\gamma}}:\mathcal{B}_{\gamma}\rightarrow \mathcal{B}_{\gamma}$ is a contraction for every $n \geq \max\{n_1, n_2\}$ with $n_2 = n_2(||D(V)||\cdot ||A(V)^{-1}||, \alpha, \gamma) \in \N$ given by
\begin{align*}
    n_2 = \left\lceil{
    \frac{1}{||D(V)||\cdot ||A(V)^{-1}||}\log\left(\frac{1}{\alpha^2(1+2\gamma)}\right)
    }\right\rceil + 1.
\end{align*}
Indeed, observe that for any $u,v\in \mathcal{B}_{\gamma}$ we have
\begin{align*}
    &\norm{\phi_{n,J}(u) - \phi_{n,J}(v)}\\
    &\leq
    \norm{(D(J) - uB(J))D(V)^nuA(V)^{-n} - (D(J) - vB(J))D(V)^nvA(V)^{-n}}\lVert A(J)^{-1}\rVert\\
    &\leq \alpha\left(
    \norm{D(J) - uB(J)}+ \norm{B(J)}\norm{v}
    \right)\norm{u-v}\left(\norm{D(V)}\lVert A(V)^{-1}\rVert\right)^n\\
    &\leq \left[
    \alpha^2(1+2\gamma)\left(\norm{D(V)}\lVert A(V)^{-1}\rVert\right)^n
    \right]\norm{u-v}.
\end{align*}
Then, for every $n\geq \max\{n_1,n_2\}$ we have that
\begin{align*}
    Lip\left(\phi_{n,J}\right) \leq \alpha^2(1+2\gamma)\left(\norm{D(V)}\lVert A(V)^{-1}\rVert\right)^n < 1.
\end{align*}

By Banach's contraction mapping theorem, for any $n\geq \max\{n_1,n_2\}$ there exists a unique $\xi_{n,J}\in \mathcal{B}_{\gamma}$ such that
\[
\phi_{n,J}(\xi_{n,J}) = \xi_{n,J}.
\]
Define the the matrix $\varphi_{n,J}\in \M_{\kappa_1}(\R)$ given by the following expression:
\[
\varphi_{n,J} = A(J)A(V)^n + B(J)D(V)^n\xi_{n,J}.
\]
Using the fact that $\xi_{n,J}$ is fixed by $\phi_{n,J}$ we observe that
\begin{align*}
    JV^n\cdot
    \begin{pmatrix}
    x\\
    \xi_{n,J}\cdot x
    \end{pmatrix}
    &=
    \begin{pmatrix}
    A(J)A(V)^n\cdot x + B(J)D(V)^n\xi_{n,J}\cdot x\\
    C(J)A(V)^n\cdot x + D(J)D(V)^n\xi_{n,J}\cdot x.
    \end{pmatrix}\\
    &=
    \begin{pmatrix}
    \varphi_{n,J}\cdot x\\
    \xi_{n,J} \varphi_{n,J}\cdot x
    \end{pmatrix}.
\end{align*}
Hence,
\[
JV^n(G(\xi_{n,J})) = G(\xi_{n,J}),
\]
and
\begin{align*}
    \varphi_{n,J} = \pi_{\xi_{n,J}}\cdot JV^n|_{G(\xi_{n,J})}\cdot\pi^{-1}_{\xi_{n,J}}.
\end{align*}
Note also that if we set 
\begin{align*}
    n_3 = \left\lceil{
    \frac{1}{\lVert{D(V)}\rVert\lVert{A(V)^{-1}}\rVert}\log\frac{\delta}{2\gamma\alpha}
    }\right\rceil + 1,
\end{align*}
$n_3 = n_3(\lVert{D(V)}\rVert\lVert{A(V)^{-1}}\rVert, \delta, \gamma,\alpha) \in\N $ and consider $n_0^+ = \max\{n_1,n_2,n_3\}$, using $\ref{14621.1}$, we have that for $n\geq n_0^+$,
\begin{align*}
    \norm{\varphi_{n,J}A(V)^{-n} - A(J_0)} &= \norm{B(J)D(V)^n\xi_{n,J}A(V)^{-n} + A(J) - A(J_0)}\\
    &\leq \norm{B(J)}\norm{\xi_{n,J}}\left(\lVert{D(V)}\rVert\lVert{A(V)^{-1}}\rVert\right)^n + \frac{\delta}{2}\\
    &\leq \gamma \alpha\left(\lVert{D(V)}\rVert\lVert{A(V)^{-1}}\rVert\right)^n + \frac{\delta}{2}\\
    &< \frac{\delta}{2} +\frac{\delta}{2} = \delta.
\end{align*}

Applying the same ideas to he product $J^{-1}V^{-n}$ (observe that $J^{-1}V^{-n}$ is conjugated by $V^n$ to $V^{-n}J^{-1}$) we find a non linear operator operator $\Psi: \M_{\kappa_1\times\kappa_2}(\R)\rightarrow \M_{\kappa_1\times\kappa_2}(\R)$ given by
\begin{align*}
    \Psi(u) = B(J^{-1})D(J^{-1})^{-1} + (A(J^{-1}) - uC(J^{-1}))A(V)^{-n}uD(V)^nD(J^{-1})^{-1}.
\end{align*}
Choosing $\beta$, $\alpha$ and $\gamma$ as above we can find $n_0^-\in \N$ such that $\Psi$ has an unique fixed $\hat{\eta}_{n,J} \in \M_{\kappa_1\times\kappa_2}(\R)$, $\norm{\hat{\eta}_{n,J}} \leq \gamma$, for every $n\geq n_0^-$ and $J\in \GL_d(\R)$ such that $\norm{J-J_0} \leq \beta$. Hence, by construction, we have
\[
J^{-1}V^{-n}(G(\hat{\eta}_{n,J})) = G(\hat{\eta}_{n,J}).
\]
Define $\eta_{n,J} := A(V)^{-n}\hat{\eta}_{n,J}D(V)^n \in \M_{\kappa_1\times\kappa_2}(\R)$. Then,
\[
\norm{\eta_{n,J}}\leq \gamma\left(\lVert{A(V)^{-1}}\rVert\lVert{D(V)}\rVert\right)^n
\]
and
\[
(JV^n)^{-1}(G(\eta_{n,J})) = V^{-n}\cdot(J^{-1}V^{-n})(G(\hat{\eta}_{n,J})) = V^{-n}(G(\hat{\eta}_{n,J})) = G(\eta_{n,J}).
\]
Moreover, by an appropriate choose of $n_0^-$ and by \eqref{14621.2} we have that if we set
\[
\psi_{n,J} := \pi_{\eta_{n,J}}\cdot (JV^n)^{-1}|_{G(\eta_{n,J})}\cdot \pi^{-1}_{\eta_{n,j}} = Q^{-n}C(J^{-1})\eta_{n,J} + Q^{-n}D(J^{-1}),
\]
then
\[
\norm{Q^n\psi_{n,J} - D(J_0^{-1})} < \delta.
\]
Taking $n_0 > \max\{n_0^+, n_0^-\}$ we have that for every $n\geq n_0$ and for every $J\in \GL_d(\R)$ with $\norm{J-J_0}\leq \beta$,
\[
G(\xi_{n,J}) \oplus G(\eta_{n,J}) = \R^d,
\]
which concludes the proof of the Lemma.
\end{proof}

The next Lemma is a direct corollary of Lemma \ref{11621.3} and gives the decomposition of the spectrum of matrices of the form $JV^n$ for $n$ sufficient large.
\begin{lemma}\label{14621.3}
In the context of Lemma \ref{11621.3} we find sequences $(X_n)_n\subset\M_{\kappa_1}(\R)$ and $(Y_n)\subset \M_{\kappa_2}(\R)$ such that for every $n\geq n_0$ the following holds
\begin{itemize}
    \item $\norm{X_n - A(J_0)} < \delta$, $\norm{Y_n^{-1} - D(J_0^{-1})}$;
    \item If $\zeta\in\sigma(X_nA(V)^n)$ and $\lambda\in \sigma(Y_nD(V)^n)$, then $|\zeta|>|\lambda|$;
    \item $\sigma(JV^n) = \sigma(X_nA(V)^n)\cup\sigma(Y_nD(V)^n)$.
\end{itemize}
\end{lemma}
\begin{proof}
Define
\begin{align*}
    X_n:= \left(
    \pi_{\xi_{n,J}}\cdot JV^n|_{G(\xi_{n,J})}\cdot \pi_{\xi_{n,J}}^{-1}
    \right)\cdot A(V)^{-n}
\end{align*}
and
\begin{align*}
    Y_n:= \left(
    \pi_{\eta_{n,J}}\cdot (JV^n)^{-1}|_{G(\eta_{n,J})}\cdot \pi^{-1}_{\eta_{n,j}}
    \right)^{-1}\cdot D(V)^{-n}.
\end{align*}
So, by Lemma \ref{11621.3} we have that, for every $n\geq n_0$,
\[
\norm{X_n - A(J_0)}< \delta \quand \norm{Y_n^{-1} - D(J_0)^{-1}}<\delta.
\]
Note also that, increasing $n_0$ if necessary, we have
\begin{align*}
\lVert{(X_nA(V)^n)^{-1}}\rVert\lVert{Y_nD(V)^n}\rVert
&\leq
\lVert{X_n^{-1}}\rVert\lVert{Y_n}\rVert\left(\norm{A(V)^{-1}}\norm{D(V)}\right)^n\\
&< 1,
\end{align*}
which implies that the second item of the Lemma. Moreover, since
\[
G(\xi_{n,J}) \oplus G(\eta_{n,J}) = \R^d,
\]
we have that
\begin{align*}
    \sigma(JV^n) &= \sigma(JV^n|_{G(\xi_{n,J})})\cup\sigma(JV^n|_{G(\eta_{n,J})})\\
    &= \sigma(X_nA(V)^n)\cup\sigma(Y_nD(V)^n).
\end{align*}
\end{proof}

\paragraph{Proof of the Proposition \ref{10621.2}:} Consider the sequence $(L_n)\subset \GL_d(\R)$, converging to $L\in \GL_d(\R)$ and $T = \diag[T_1:\cdots:T_m] \in \GL_d(\R)$ satisfying the conditions of the proposition. Remember that we are fixing the decomposition $\R^d = \R^{i_1}\times\cdots\times \R^{i_m}$ with $i_j\in \{1,2\}$ for every $j = 1,\dots, m$. We also use the following notation:
\begin{itemize}
    \item For each $j=1,\dots, m$, $\kappa_j = i_j+\cdots i_m$;
    \item For every invertible quadratic matrix $J$, $\iota(J) = J^{-1}$;
    \item For every $J\in \M_{\kappa_j}(\R)$ we denote by $A_j:\M_{\kappa_j}(\R)\rightarrow \M_{i_j}(\R)$ and $D_j:\M_{\kappa_j}(\R)\rightarrow \M_{\kappa_{j+1}}(\R)$ the canonical projections such that
    \begin{align*}
        J=
        \left(
        \begin{array}{cc}
            A_j(J) & * \\
            * & D_j(J)
        \end{array}
        \right);
    \end{align*}
    Observe that, since $T$ is diagonal we have that $A_j(T) = T_j$;
    \item $D^{(j)}:\M_d(\R)\rightarrow \M_{\kappa_{j+1}}(\R)$ is given by $D^{(j)}(J) = D_j\circ\cdots\circ D_1(J)$, $D^{(0)}(J) = J$.
\end{itemize}

For each $j = 1,\dots, m-1$ we apply Lemma \ref{14621.3} with the decomposition $\R^{\kappa_j} = \R^{i_j}\times\R^{\kappa_{j+1}}$ and the matrices $J_0 = \iota\circ D^{(j-1)}\circ\iota(L)\in \GL_{\kappa_j}(\R)$ and $V = D^{(j-1)}(T)$ to obtain that for any given $\delta_j> 0$, there exists $n_{m-j}\in \N$ and $\beta_j\in(0,\delta_j)$ such that for every $J\in \GL_{\kappa_j}(\R)$, with $\norm{J - \iota\circ D^{(j-1)}\circ\iota(L)}< \beta_j$, and for every $n\geq n_{m-j}$ there exists quadratic matrices $X_{n,J}\in \M_{i_j}(\R)$ and $Y_{n,J}\in \M_{\kappa_{j+1}}(\R)$ satisfying that
\begin{itemize}
    \item $\sigma(JD^{(j-1)}(T)^n)) = \sigma(X_{n,J}T_j^n)\cup\sigma(Y_{n,J}D^{(j)}(T)^n)$;
    \item $\norm{X_{n,J} - A_{j+1}\circ\iota\circ D^{(j)}\circ\iota(L)} < \delta_j$ and $\norm{\iota(Y_{n,J}) - D^{(j)}\circ\iota(L)} < \delta_j$.
\end{itemize}

\subparagraph{Choice of parameters:} Fix $\varepsilon_0>0$. The idea is inductively to choose the sequence of parameters $\delta_j>0$, $j =1,\dots, m-1$, starting in the last stage $\delta_{m-1}$. It is important to highlight that for each chosen $\delta_{m-j+1}$ the Lemma \ref{14621.3} give us a respective $\beta_{m-j+1}>0$ which will be used to define the next $\delta_{m-j}$ for every $j=1,\dots,m-1$.
\begin{enumerate}
    \item Choose $\delta_{m-1} = \delta_{m-1}(\varepsilon_0, D^{(m-1)}\circ\iota(L))\in (0,\varepsilon_0)$ such that for every $J\in GL_{\kappa_m}(\R) = \GL_{i_m}(\R)$, with
    \[
    \norm{\iota(J) - D^{(m-1)}\circ \iota(L)} < \delta_{m-1}
    \]
    we have
    \[
    \norm{J-\iota\circ D^{(m-1)}\circ\iota(L)}< \varepsilon_0.
    \]
    The existence of such $\delta_{m_1}$ is justified by the continuity of $\iota$ at $D^{(m-1)}\circ\iota(L)$.
    
    Assume that we have chosen $\delta_{m-1}>\cdots>\delta{m-j+1}$. Let $\beta_{m-j+1}>0$ associated with $\delta_{m-j+1}$ given by the Lemma \ref{14621.3}. Now take
    \[
    \delta_{m-j} = \delta_{m-j}(\delta_{m-j+1}, D^{(m-j)}\circ\iota (L))\in (0,\delta_{m-j+1})
    \]
    such that for every $J\in \GL_{\kappa_{m-j}}(\R)$ with
    \[
    \norm{\iota(J) - D^{(m-j)}\circ\iota(L)}< \delta_{m-j}
    \]
    we have
    \[
    \norm{J - \iota\circ D^{(m-j)}\circ\iota(L)} < \beta_{m-j+1}.
    \]
    The existence of $\delta_{m-j}$ is justified by the continuity of $\iota$ at $D^{(m-j)}\circ\iota(L)$. So, by induction we have defined $0< \delta_1<\delta_2<\cdots<\delta_{m-1} < \varepsilon_0$.
    
    \item Choose $k_0\in \N$ such that for every $k\geq k_0$
    \[
    \norm{L_k - L}< \delta_1.
    \]
\end{enumerate}

So, we can find $n_0\in \N$ such that for every $n\geq n_0$, for every $k\geq k_0$ and for every $j=1,\dots, m-1$ there exist $X^{(j)}_{n,k} = X_{n,X^{(j-1)}_{n,k}} \in \M_{i_j}(\R)$ and $Y^{(j)}_{n,k} = Y_{n,Y^{(j-1)}_{n,k}}\in \M_{\kappa_{j+1}}(\R)$, with the convention that $X^{(0)}_{n,k}=Y^{(0)}_{n,k} = L_k$, satisfying that
\begin{itemize}
    \item $\sigma(Y^{(j-1)}_{n,k}D^{(j-1)}(T)^n) = \sigma(X^{(j)}_{n,k}T_j^n)\cup\sigma(Y^{(j)}_{n,k}D^{(j)}(T)^n)$;
    \item $\norm{X^{(j)}_{n,k} - A_{j+1}\circ\iota\circ D^{(j)}\circ\iota(L)} <\delta_j$ and $\norm{\iota(Y^{(j)}_{n,k}) - D^{(j)}\circ\iota(L)} < \delta_j$.
\end{itemize}
 Observe that by the choice of $\delta_j$ this implies that
    \[
    \norm{Y^{(j)}_{n,k} - \iota\circ D^{(j)}\circ\iota(L)} < \beta_j.
    \]
Hence, for every $n\geq n_0$ and for every $k\geq k_0$ we have that
    \begin{align}\label{15621.3}
        \sigma(L_kT^n) = \bigcup_{j=1}^m\sigma(X^{(j)}_{n,k}T^n_j),
    \end{align}
    and
    \begin{align}\label{15621.1}
        \norm{X^{(j)}_{n,k} - A_{j+1}\circ\iota\circ D^{(j)}\circ\iota(L)}<\varepsilon_0,
    \end{align}
for every $j \in \{1,\dots, m-1\}$, with the convention that $A_m\circ\iota\circ D^{(m-1)}\circ\iota(L)) = \iota\circ D^{(m-1)}\circ\iota(L)$ and that $X^{(m)}_{n,k} = Y^{(m-1)}_{n,k}$ and so,
\begin{align}\label{15621.2}
    \norm{X^{(m)}_{n,k} - \iota\circ D^{(m-1)}\circ\iota(L)} < \beta_j < \varepsilon_0.
\end{align}
Consider the set
\[
I = \{j\in \{1,\dots,m\};\ i_j = 2\}.
\]
Using polar form decomposition theorem for each $j\in I$, for every $n\geq n_0$ and for every $k\geq k_0$ there exist positive definite matrices $P^{(j)}$, $P^{(j)}_{n,k}$ and rotations $R_{\alpha^{(j)}}$, $R_{\alpha^{(j)}_{n,k}}$ such that 
\begin{align*}
    A_{j+1}\circ\iota\circ D^{(j)}\circ(L) = P^{(j)}R_{\alpha^{(j)}} \quand X^{(j)}_{n,k} = P^{(j)}_{n,k}R_{\alpha^{(j)}_{n,k}}.
\end{align*}
If we choose $\varepsilon_0>0$ sufficient small, by inequalities \eqref{15621.1} and \eqref{15621.2}, there exists $\hat{\varepsilon}>0$ small, such that, for every $\varepsilon\in (0,2\hat{\varepsilon})$, we have that the eigenvalues of $P^{(j)}_{n,k}R_{\varepsilon}$ are real and simple, moreover $\alpha^{(j)}_{n,k}$ is $\varepsilon$-close to $\alpha^{(j)}$ for every $n\geq n_0$ and every $k\geq k_0$.

Consider the arithmetic progression $(an+b)_n$ as in the statement of the Proposition \ref{10621.2} and write
\[
P^{(j)}_n := P^{(j)}_{an+b, n}, \alpha^{(j)}_n = \alpha^{(j)}_{an+b,n},
\]
and
\begin{align*}
X^{(j)}_n = X^{(j)}_{an+b, n} = P^{(j)}_nR_{\alpha^{(j)}_n}.
\end{align*}
Then, increasing $n_0$ if necessary,
\begin{align*}
    X^{(j)}_nT^{an+b}_j = P^{(j)}_nR_{\alpha^{(j)}_n + (an+b)\theta_j}.
\end{align*}
Since $\theta = (\theta_{i_j})_{j\in I}\in \R^{\# I}$ is rationally independent we can consider a subsequence $(n_l)_l\subset \N$, $n_l\to\infty$, such that
\[
a\theta n_l + b\theta \to -\alpha = -(\alpha^{(j)})_{j\in I}.
\]
So, defining $\varepsilon_{n_l} = \alpha^{(j)}_{n_l} + (an_l+b)\theta_j$ we have that there exists $l_0\in\N$ such that $\varepsilon_{n_l} \in (0,2\hat{\varepsilon})$ for every $l\geq l_0$. Therefore,
\[
X^{(j)}_{n_l}T^{an_l+b}_j = P^{(j)}_{n_l}R_{\varepsilon_{n_l}}
\]
has real and simple spectrum for every $j=1,\dots,m$ and so, by equation \eqref{15621.3},
\begin{align*}
    L_{n_l}T^{an_l+b}
\end{align*}
has real and simple spectrum.\qed

\subsection*{Acknowledges:} The authors will like to thank Mauricio Poletti for the suggestions that help to improve the writing of the manuscript. J. B. was supported by CAPES and FCIENCIAS.ID.

\bibliographystyle{abbrv}
\bibliography{head.bib}

\information

\end{document}